\documentclass[12pt]{article}

\usepackage{latexsym,amssymb}
\usepackage[T1]{fontenc}
\usepackage[dvips]{graphicx}
\newcommand{\sma}{\left(\begin{array}}
\newcommand{\fma}{\end{array}\right)}
\newcommand{\rr}{\mathbb R}
\newcommand{\z}{\mathbb Z}
\newcommand{\q}{\mathbb Q}
\newcommand{\n}{\mathbb N}
\newtheorem{lem}{Lemma}[section]

\newtheorem{co}[lem]{Corollary}
\newtheorem{thm}[lem]{Theorem}
\newtheorem{prop}[lem]{Proposition}

\newtheorem{qu}[lem]{Question}
\newenvironment{proof}{\textbf{Proof.}}{\newline\hspace*{\fill}{$\Box$}}

\begin{document}
\title{Proving finitely presented groups are large by computer}
\author{J.\,O.\,Button\\
Selwyn College\\
University of Cambridge\\
Cambridge CB3 9DQ\\
U.K.\\
\texttt{jb128@dpmms.cam.ac.uk}}
\date{}
\maketitle
\begin{abstract}
We present a theoretical algorithm which, given any finite presentation 
of a group as input, will terminate with answer yes if and only if the 
group is large. We then implement a practical version of this algorithm 
using Magma and apply it to a range of presentations. Our main focus
is on 2-generator 1-relator presentations where we have a complete
picture of largeness if the relator has exponent sum zero in one generator and
word length at most 12, as well as when the 
relator is in the commutator subgroup and has word length at most 18.
Indeed all but a tiny number of presentations define large groups.
Finally we look at
fundamental groups of closed hyperbolic 3-manifolds, where the algorithm
readily determines that a quarter of the groups in the Snappea closed
census are large.
\end{abstract}

\section{Introduction}

A finitely presented group $G$ is said to be large if there exists a finite
index subgroup possessing a surjective homomorphism to a non abelian
free group (without loss of generality we can assume this is the free group
$F_2$ of rank 2). This notion is used more widely for finitely
generated groups, but in this paper all groups considered will be
finitely presented. In any case it is a strong property which implies a
whole host of consequences: the most relevant one here is that $G$ has
infinite virtual first Betti number (meaning that $G$ has finite index
subgroups with arbitrarily large first Betti number). It is unknown if
there is an algorithm which takes as input a finite presentation and tells
us whether or not the group defined by that presentation is large. The
two standard methods used in establishing unsolvability are to show that the
property is Markov or is incompatible with free products (see \cite{ls}
Chapter IV Section 4), but the property of being large is neither of these,
and nor is its negation. However there is a partial algorithm for
largeness as pointed out by I.\,Kapovich
which is guaranteed to terminate with the answer yes if the
input presentation gives rise to a large group but which will not terminate
otherwise. In \cite{htrs} it is noted that there is a partial algorithm
that will tell if a presentation has a free quotient of rank $k$ (but
which will not terminate otherwise). Therefore one
can begin running this on a given finite presentation for a group $G$.
Although this might run for ever, even if $G$ is large, one immediately
starts elsewhere a separate new process to evaluate the finite index 
subgroups of $G$, along with a presentation for each subgroup, and then 
further starts the free quotient algorithm many times in parallel on each
finite index subgroup. But nobody would ever want to implement this: the
free quotient algorithm is described in \cite{htrs} as being totally
impractical for long presentations, so we would be running an extremely
slow process a vast number of times simultaneously.

In this paper we describe an alternative partial algorithm for largeness
that one might actually want to implement. Furthermore we do this by
writing a program in Magma and we apply it to a considerable range of
presentations. It is based on \cite{meis} Theorem 2.1 which says that
if $G$ is a finitely presented group having a homomorphism $\chi$ onto
$\z$ such that the Alexander polynomial $\Delta_{G,\chi}$ relative to
$\chi$ is zero, or zero modulo a prime $p$, then $G$ is large. This in
turn is based on a result of Howie in \cite{hw}. The Alexander polynomial
can be calculated reasonably efficiently given a presentation for $G$
(the time consuming part of the process for long presentations being
the calculation of large determinants) so it would seem that one needs to
go through the finite index subgroups $H$ and check each Alexander polynomial
in turn. This is essentially the idea, but the problem is that when 
$\beta_1(H)\geq 2$ we have infinitely many homomorphisms from $H$ onto
$\z$. Therefore we need to establish that we can determine by a finite process
whether or not there exists one of these
homomorphisms with zero Alexander polynomial. Indeed it is not surprising
that it will often be a finite index subgroup of $G$ with first Betti
number at least two which allows us to conclude largeness, so we will
want to be able to do this quickly. We describe our algorithm in Section
2 and then we go on to report our results when running it on a computer.
Our first application in Section 3 is to deficiency 1 groups, that is
finitely presented groups where the given presentation has one more
generator than relator. We have a particular interest in the largeness of
deficiency 1 groups: although they are not always large, unlike groups of
deficiency two or above, we have results in \cite{meis} and \cite{megg}
that they are often large. Therefore it would be good to have experimental
evidence of this as well. Moreover a finite index subgroup of a deficiency
1 group also has deficiency 1.

In Section 3 we only look at 2 generator 1 relator presentations. This is
not because of any limitations of the method (indeed when applying the
program to such a presentation, each finite index subgroup that we work
with will have a deficiency 1 presentation which is rarely 1 relator) but
because we are able to cover a lot of ground in this special case. We look
at presentations where one of the two generators has zero exponent sum in
the relator, as any 2 generator 1 relator presentation can be put in such
a form using an automorphism of $F_2$. For these presentations where the
relator is of length 12 or less, we have a definitive result. It can be
summarised by saying that the vast majority of presentations are large,
the presentations which are not large are listed in Tables 1 and 2 and
all of these are groups which were already known not to be large.
(Sometimes these groups are given by non standard presentations but they
were already known too.) Amongst these presentations $\langle a,t|r\rangle$
where $r$ has exponent sum 0 in $t$, we have those where the relator is
of height 1, which means that $r$ can be written purely in terms of 
$a^{\pm 1}$ and $ta^{\pm 1}t^{-1}$. We have a result in \cite{megg} which
gives a much more efficient criterion for largeness of such a presentation
and so we consider height 1 relators of length at most 14. For these we
are able to list all of the non large presentations except for two where
we believe that they are not large but we do not recognise them as groups that
are already known not to be large. Also in \cite{megg} we gave an example of
a large word hyperbolic group which is of the form $F_k\rtimes_\alpha\z$,
but the automorphism $\alpha$ was reducible. Here we are able with use of the
computer to give the first example of a large word hyperbolic
group of this form where all
non-trivial powers of the automorphism are irreducible.

We pay further attention to where the relator is in the commutator subgroup
$F_2'$. These are very often large (unless the group is $\z\times\z$ in
which case the relator, if cyclically reduced, is just a commutator of the
two generators) but in \cite{bmgg} an example was given of a non large group
where the particular relator $r_0$ has length 18. We again have a definitive
result in that if $r$ is a cyclically reduced word in $F_2'$ 
and has length at most 18 then our program
shows that either the group $G$ is large, or $r$ has length 4 so that
$G=\z\times\z$, or $r$ has length exactly 18 and is just a cyclic
permutation of $r_0$ or $r_0^{-1}$. 

Our final application in Section 4 is to a class of groups which have
deficiency 0 rather than 1. These are the fundamental groups of closed
orientable hyperbolic 3-manifolds and an open question asks if they are
always large. We already have a sample of over 10,000 examples to work
with which is the census from the program Snappea \cite{cen} and it is also
available as a Magma database where the fundamental groups of these
3-manifolds are given. It is known by recent results that if a closed
hyperbolic 3-manifold is arithmetic and has positive virtual first Betti
number then we have largeness. However positive virtual first Betti number
for all 3-manifolds in the census was established using computational methods
in \cite{dnth}. Therefore all arithmetic 3-manifolds in the census are
known to be large but there are no such results yet in the non arithmetic
case. Our requirement for largeness of having a zero Alexander polynomial
means that only subgroups with positive first Betti number can be of use in
satisfying this condition. Consequently in order to help with the computations
we restrict attention to the groups in the census having a subgroup with
positive first Betti number of index at most 5. There are 2856 such groups
in the census which is over a quarter of the total, and they can be found
quickly using Magma. Our program proves that most of these groups are
large; indeed we are left with 116 groups in Table 4 where we did not
establish this within the limited running times. Moreover there are 132
groups in the census which themselves have positive first Betti number,
and a further 305 which have finite first homology but an index 2 subgroup
with positive first Betti number. None of these are left over in the list
so we can conclude that a 3-manifold in the closed census
with positive first Betti number or which has a double cover with
positive first Betti number also has large fundamental group.

\section{Description of the algorithm}

Given a finitely presented group $G$ that we wish to prove is large,
we will be done if we can find a finite index subgroup $H\leq_f G$
and a surjective homomorphism from $H$ to $\z$ such that the
Alexander polynomial $\Delta_{H,\chi}$ with respect to $\chi$ is the
zero polynomial by \cite{meis} Theorem 2.1.
Here the Alexander polynomial can either have
coefficients in $\z$ or modulo a prime $p$ (we will sometimes refer to
the former as the mod 0 case). However we also have:
\begin{prop} If $H$ is a finitely presented group which has a surjective
homomorphism $\theta$ to a non-abelian free group $F_n$ of rank $n\geq 2$ then
we have homomorphisms $\chi$ from $H$ onto $\z$ with $\Delta_{H,\chi}=0$.
\end{prop}
\begin{proof}
There are homomorphisms $\chi$ onto $\z$ which factor through $F_n$; take any
one of these so that $\chi=\tilde{\chi}\theta$. Then $\theta$ sends ker $\chi$
onto ker $\tilde{\chi}$, but the free group $F_n$ has no non-trivial finitely
generated normal subgroups of infinite index, so ker $\tilde{\chi}$ is an 
infinitely generated free group with $\beta_1(\mbox{ker }\tilde{\chi};\q
)=\infty$. This implies that $\Delta_{H,\chi}=0$ by \cite{meis} Corollary 2.2.
\end{proof}

Therefore this condition of having a finite index subgroup with zero Alexander
polynomial relative to some homomorphism is both necessary and sufficient for
a finitely presented group to be large. 
Recall that there is an algorithm which takes as input a finite presentation
and a positive integer $n$ and which outputs all the (finitely many)
subgroups $H$ having index $n$ in the group $G$ defined by the presentation. 
This is shown in \cite{disch} and is based on the Todd-Coxeter coset 
enumeration process. The output for each $H$ is a list of generators of $H$ and
a coset table for the right regular action of $G$ on the cosets of $H$. This
allows us by the Reidemeister-Schreier rewriting process to give a finite
presentation for $H$. In this section we describe an
algorithm that, given a finite presentation of a group $G$,
works out whether or not there is a homomorphism $\chi$ and $p$
(or 0) such that $\Delta_{G,\chi}=0$ mod $p$ (or mod 0). 
Consequently we then have a partial algorithm for largeness by applying
this to each finite index subgroup of $G$ in turn.

It is important for a practical
algorithm to be able to deal with all of these cases simultaneously:
whilst a large group $G$ will always have some finite index subgroup
$H$ where $\Delta_{H,\chi}$ is 0 mod 0, the index $[G:H]$ could be
much larger than $[G:S]$ where $\Delta_{S,\chi}$ is, say, 0 mod 2.
However it is not efficient to run our algorithm on one subgroup $S$
modulo successive primes as we
need to know when we can give up on $S$ and move on to another subgroup.

We note that our approach depends markedly on the first Betti number
$\beta_1(G)$. If this is zero then there are no homomorphisms $\chi$
onto $\z$ and we must reject $G$ immediately. If $\beta_1(G)=1$ then we
have just one $\chi$ (up to sign) and the evaluation of $\Delta_{G,\chi}$
is essentially a straight calculation, although we need to work
modulo all primes simultaneously.
However if $\beta_1(G)\geq 2$ we have
infinitely many $\chi$. Thus we have the advantage of many chances to
find a $\chi$ with $\Delta_{G,\chi}$ zero but the disadvantage that
we cannot test all of the $\chi$ individually and so we need a method
of narrowing down our search.

Let $\beta_1(G)=b$. If $b=1$ then the Alexander polynomial
$\Delta_{G,\chi}$ is an element (defined up to multiplication by units)
of the Laurent polynomial ring $\z[t^{\pm 1}]$ which is the integer
group ring of $\z$.
For a particular $\chi$ we can find $\Delta_{G,\chi}$
in the following way: given a presentation
$\langle x_1,\ldots ,x_n|r_1,\ldots ,r_m\rangle$ of $G$, we apply 
Fox's free differential calculus to form the Alexander matrix $A$
which is an $m\times n$ matrix with entries in 
$\z[t^{\pm 1}]$. We then calculate all the $(n-1)\times
(n-1)$ minors of $A$: namely the determinants of the submatrices of $A$
formed by deleting one column and the necessary number of rows to
make the submatrices square, so it will be $m-n+1$ rows and there will
be $n\times\sma{c} m\\{m-n+1}\fma$
different minors. Then $\Delta_{G,\chi}$ is
defined to be the highest common factor of these minors and
consequently it is 0 (or 0 mod $p$) if and only if all of the minors are
0 (or 0 mod $p$). 

If $I$ is a subset of size $m-n+1$ chosen from $\{1,2,\ldots ,m\}$ then
we use $M_j^I$ to denote the minor with the rows in $I$ removed along with
the $j$th column. If $\chi(x_1)=a_1,\ldots ,\chi(x_n)=a_n$ then we have
the identity $M_j^I(1-t^{a_i})=M_i^I(1-t^{a_j})$. Consequently if $a_j=0$
then $M_j^I=0$ anyway (by taking $i$ for which $a_i\neq 0$) and so there is
no point in calculating this minor, but otherwise we have $M_j^I=\delta^I
\psi_{a_j}(t)$ where $\delta^I$ is independent of $j$ and $\psi_k(t)$ is
equal to $(1-t^k)/(1-t)$. This is because only $1-t$ can
divide all of $1-t^{a_1},\ldots ,1-t^{a_n}$. Consequently if we take the
first column $j$ such that $\chi(x_j)\neq 0$ then we have that 
$\Delta_{G,\chi}$ is the highest common factor of the $\delta^I$ where $I$
is varied over all possible subsets, thus meaning that we have reduced the
number of minors that need to be calculated to $l=\sma{c} m\\m-n+1\fma$
minors, thus reducing the work by a factor of $n$. 

Therefore in the case $b=1$ we merely form the Alexander matrix $A$ and,
on taking the above $j$ we calculate in turn the minors
$M_j^{I_1},\ldots ,M_j^{I_l}\in\z[t^{\pm 1}]$.
If they are all the zero polynomial then
we have proved largeness. If $M_j^{I_i}$ is the first non-zero minor then we
look at its content: if this is 1 then we stop immediately with the
answer no, otherwise we calculate $M_j^{I_{i+1}}$ and let $c$ be
gcd (content($M_j^{I_i}$),content($M_j^{I_{i+1}}$)). 
We continue replacing $c$ with
the  gcd of $c$ and the content of the next minor and stop with no if $c$
becomes 1, otherwise when we reach the last minor
$M_j^{I_l}$ we have proved that 
$\Delta_{G,\chi}$ is 0 modulo any prime that divides $c$ so we have largeness.

In fact for $\beta_1(G)=1$ we will not find that all of the minors are
identically zero. This can be seen because they continue to be zero
if the Alexander matrix is evaluated at $t=1$, but this is just a
presentation matrix for the abelianisation $G/G'$ which is of the form
\[C_{d_1}\times\ldots\times C_{d_k}\times\z\qquad\mbox{ for }d_1|d_2|
\ldots |d_k.\]
Thus the first elementary ideal of $A|_{t=1}$, which is an invariant of the
abelian group, is $d_1\ldots d_k$. Consequently any primes dividing the
content of a minor $M_j^I$ 
must also divide $M_j^I|_{t=1}$ and hence $d_k$. This
implies two improvements: we should always reject groups whose
abelianisation is just $\z$, and more generally we can work in the ring
$(\z/d_k\z)[t^{\pm 1}]$ rather than $\z[t^{\pm 1}]$ 
when evaluating determinants if this will be significantly quicker.
(In fact when using Magma, we always found that it was better to
work modulo the primes dividing $d_k$ separately rather than with respect
to a composite modulus, and this did give an advantage of speed over the
characteristic zero case.)
 
We now proceed to describe a theoretical algorithm in the case when
$b\geq 2$. In order to consider all $\chi$ together, we can replace
the ring $\z[t^{\pm 1}]$ above with the ring 
$\z[t_1^{\pm 1},\ldots ,t_b^{\pm 1}]$
which is the integral group ring of the free abelianisation $ab(G)=\z^b$
of $G$. We also have a Fox calculus in this case
(see \cite{memap} Section 2 for an exposition in line with our approach
here) and so can form a more general Alexander matrix $B$ with entries
in $\z[ab(G)]$ and corresponding minors $N_j^I$. As any surjective
$\chi:G\rightarrow\z$ will factor through the natural map
$\alpha:G\rightarrow ab(G)$ and so can be written as $\tilde\chi\alpha$,
we have that the Alexander matrix $A$ with respect to any given
homomorphism $\chi$ is just $B$ evaluated at $\tilde\chi$ and
consequently the minors $M_j^I$ are equal to $N_j^I|_{\tilde\chi}$.

Thus we calculate all the minors $N_j^I$ which are multivariable
polynomials with coefficients in $\z$ 
and we now need to consider which $\chi$ will
make all the minors vanish. We do this by regarding the minor $N_j^I$
as a finite subset of
lattice points in $\z^b$ with each point weighted by a non-zero integer, 
where each monomial that appears in $N_j^I$ with a non-zero
coefficient is a lattice point and the coefficient is the weight. (The
ambiguity of units just means that we can shift $N_j^I$ by unit 
translations.) We picture evaluation of
$N_j^I$ at $\tilde\chi$ in the following way: we extend
$\tilde\chi$ to an affine map 
$\phi:\mathbb R^b\rightarrow\mathbb R$. Then for $x\in\mathbb R$ we know that
$\phi^{-1}(x)$ is a hyperplane and $N_j^I$ is zero on evaluation precisely
when the following condition is satisfied: for all $m\in\z$ with $\phi^{-1}(m)
\cap N_j^I\neq\emptyset$, we require that the sum of the weights 
corresponding to the points of $N_j^I$ in this hyperplane $\phi^{-1}(m)$ is
zero. Let us refer to this situation as ``$N_j^I$ cancels along parallel
hyperplanes of constant $\chi$''.

Starting with $N_1$ (which is defined to be the first minor 
we calculate that happens not to be identically zero) we
take any lattice point $\bf x$ in $N_1$. Now $\bf x$ must cancel with other
lattice points in $N_1$ on evaluation so we take each other lattice point
$\bf y$ in turn and join $\bf x$ and $\bf y$ by a line ${\bf x}+U$ where
$U$ is a one dimensional subspace of $\rr^b$. We now take the quotient
vector space $V=\rr^b/U$ of dimension $b-1$ and use the quotient map $q$ to
regard $N_1$ as a finite subset of $V$ with with new weights obtained by 
summing within the translates of $U$. We now pick a basepoint of $q(N_1)$,
draw a line from it to the other points of $q(N_1)$ and continue
recursively. This process stops either because an image of $N_1$ is the
zero polynomial in dimension $d$, or we reach $d=1$ with a non zero
polynomial. In the latter case we can reject this $\chi$ (at least if we
are working modulo 0) but in the former we move on to the next minor
$N_2$. We first apply to $N_2$ the composition of the quotient maps
applied so far and then we continue
in the same way, by picking a basepoint for this image of $N_2$
and another point to join it to and proceed as before, 
thus reducing $d$ (or $N_2$ might become 0 by
dimension $d$, in which case we move on to the next minor and so on).
Eventually we have either found that all the minors have vanished at
dimension $d\geq 1$, in which case any $\chi$ which factors through the
composition of quotient maps applied so far will have $\Delta_{G,\chi}
=0$, or we reach dimension 1 and conclude that we have failed to reach
a suitable $\chi$ along this path, so we must backtrack.

This will do as a starting algorithm, however there are
areas where serious improvement would be hoped for. First the above
only works modulo 0 and we want to be able to take advantage of
occasions where $\Delta_{G,\chi}$ is zero modulo a prime (indeed during
the significant computations described in Section 3 it was nearly always
witnessed that $\Delta_{G,\chi}$ was zero modulo one or two small
primes but not mod 0 and in many cases it was 
just zero mod 2). More seriously, although the above argument will work
on minors with few coefficients, the branching nature of this
approach means it will be extremely expensive for large minors (and indeed
large numbers of minors). Moreover we need to calculate all of the minors
but evaluating determinants over multivariable polynomial rings takes
considerably longer than in $\z[t^{\pm 1}]$ or $(\z/p\z)[t^{\pm 1}]$ 
so we will gain if we can work more in the latter rings.

We now describe a variation on the above in the special case $\beta_1(G)=2$.
This is the algorithm we implemented in Magma and we found it performed
quickly. It seemed that determinants of large matrices (say 12$\times$12
and larger) were much quicker to calculate over $\z[t^{\pm 1}]$ than
over multivariable polynomial rings so the following approach aims to
do as little of the latter as possible.

As in the case $b=1$ there is a similar identity between different
minors for general $b$.
We have $N_j^I(1-\alpha(x_i))=N_i^I(1-\alpha(x_j)$ where
for $1\leq i\leq n$ we have that $\alpha(x_i)$ and $N_i^I$ are 
elements of $\z[t_1^{\pm 1},\ldots ,t_b^{\pm 1}]$. This implies that
$N_j^I=\delta^I(1-\alpha(x_j))$ because any common factor of
$1-\alpha(x_1),\ldots ,1-\alpha(x_n)$ would remain so when evaluated under
all $\tilde\chi$. Consequently we build the Alexander matrix $B$ and in doing
so we record the vectors $v_1,\ldots ,v_n\in\z^2$ where 
$v_j=(a_j,b_j)$ for $\alpha(x_j)=x^{a_j}y^{b_j}\in\z[x^{\pm 1},y^{\pm 1}]$.

We first check the two special homomorphisms $\chi_x$ (and $\chi_y$)
which means that we are setting $x$ (and then $y$) equal to 1 in $B$
to form the evaluated matrix $B_x$ (or $B_y$).
We skip across the columns until we find a generator $x_j$ with the
second component of $v_j$ being non zero and delete the $j$th column
(note that if it were zero then any minor calculated with this column
removed would be zero anyway).

We then run through all possible choices of rows that make the
resulting matrix square when this choice is deleted. On taking our 
first choice and evaluating the
determinant, we look at and record the content $c_1$ of the resulting 
single variable polynomial. If $c_1=1$ then we can reject $\chi_x$ as a
homomorphism making all the minors vanish modulo some prime, so we move
on to $\chi_y$ and do the same. Otherwise calculate another determinant 
with the $j$th column still removed but a different selection of rows
and find the content $c_2$ of this. We then take $c=$gcd$(c_1,c_2)$ and
continue calculating minors along with their contents,  then we
update $c$ by taking the gcd of it and the content of the new minor. 
We stop calculating minors when the gcd becomes 1 and
reject $\chi_x$ (and then we try $\chi_y$)
or we may find on calculating the content of
all minors with the $j$th column removed that they have a common factor, 
in which case we have proved largeness.

At this point we have tried only two homomorphisms but we can now use a
form of Chinese Remainder Theorem to test the rest. Consider the
homomorphism $\chi$ sending $(x,y)$ to $(t^l,t^m)$ where gcd$(l,m)=1$.
The monomial (or point in $\z^2$) $x^ay^b$ is sent to $t^{la+mb}$ so in
order for any minor
$N(x,y)=\sum c_{a,b}x^ay^b$ to be 0 under evaluation, we need
each sum of the $c_{a,b}$s over $(a,b)$ such that $la+mb$ is constant to
be zero. If this happens for a particular $\chi$ and a prime $p$ divides
$m$ (so does not divide $l$) then the set of points $(a,b)$ making up
each sum is such that $a$ is constant mod $p$. This works for prime
powers $p^i$ too. Consequently we calculate the first full minor $N(x,y)$
such that $N(x,1)$ is not identically zero and
we will know from above which minor this is by looking in our list of
contents. We then consider the polynomial $P(x,y)$ which is the quotient
of $N(x,y)$ by $1-\alpha(x_j)$ where $j$ is the column that we deleted.

Then for each $p^i$ at most the degree of $P(x,y)$ as a polynomial
in $x$, we form the ``wrapped up polynomial'' $\sum c_{a,b}x^{\overline{a}}
y^b$ where $\overline{a}$ is $a$ mod $p$ (so it is
of degree at most $p^i-1$ in $x$) and we see whether this new
polynomial is zero.
We start with 2, then powers thereof, dropping out if a power fails this
test and moving on to the next prime. We then record in some set the
prime powers that pass as possible factors of $m$. The purpose of
first checking $\chi_y$ (and $\chi_x$) is that we now know $m\neq 0$ 
(and $l$) so
we will not have many possibilities for $m$. We next do the same for $l$
and together try out all of these $\chi$ to see if any of them work. This
can be done efficiently by evaluating the Alexander matrix under each potential
homomorphism so that the entries are in $\z[t^{\pm 1}]$. We then find a
column $j$ such that $1-\chi(x_j)$ is not zero in $\z[t^{\pm 1}]$ and
remove that column. We can then calculate the minors over all
choices of rows to delete and this can be done quickly as they are
single variable polynomials. We have largeness if all of them vanish.

We can also see if there exists a prime $p$ such that
all of the minors vanish modulo $p$ on evaluation 
under a particular $\chi$. To do this we can proceed as above, recording the
content $c_i$ of each of the single variable polynomials $N_i(x,1)$ (and
similarly for $N_i(1,y)$). We then take $N(x,y)$, where $N(x,1)$ has non zero
content, and form $P(x,y)$ as before. 
Next we create each wrapped up polynomial of 
$P(x,y)$ with respect to a prime power.
However we do not just look to see if this polynomial vanishes but if it has
content which is not equal to 1, in which case we record the prime power
and this content. We can then find as before candidates for $l$ and $m$ and 
check the appropriate homomorphisms to see if all minors vanish modulo some
integer under such a homomorphism by examining the contents of these single
variable polynomials and taking a gcd as we go along. However 
note that if the content $c$ of $N(x,1)$ is not equal to 1 then the same 
will be true for $N(x,y)$ and $P(x,y)$,
thus every prime power will pass the above test
modulo $c$ when we wrap up $P(x,y)$. 
We deal with this by removing from the contents of the wrapped up
polynomials any primes which
divide $c$. We then check the candidate homomorphisms $\chi$ obtained
from this process to see if all minors evaluated at $\chi$ are 0 mod $n$ 
for some $n$. If not then we finish by going back to the prime 
factors $r$ of $c$ and for
each one we calculate in full another minor $N'(x,y)$ such that $N'(x,1)$
has content coprime to $r$ (and the same for $y$ in place of $x$). 
Moreover we will know which minor to choose from our list of contents
$c_1,\ldots ,c_l$. We then run the process above to obtain candidate
homomorphisms but this time we work modulo $r$ throughout. 
Finally we have determined
whether there is a homomorphism $\chi$ and a prime $p$ such that all
minors evaluated at $\chi$ are 0 mod $p$.
 
We now present a variation for $b=\beta_1(G)\geq 3$ which is particularly
useful for showing that there is no homomorphism with zero Alexander
polynomial. For any prime $p$ we can ``box'' a given multivariable
polynomial minor 
\[N(t_1,\ldots ,t_b)=\sum_{{\bf v}\in\n^b}c_{\bf v}t_1^{v_1}\ldots 
t_b^{v_b}\]
where ${\bf v}=(v_1,\ldots ,v_b)$ by regarding the exponents $v_i$ as
integers mod $p$. If there is a homomorphism $\chi=(k_1,\ldots ,k_b)$
such that $N(t^{k_1},\ldots ,t^{k_b})$ is 0 (mod 0 or mod some $n$) then this
is still true for the ``boxed'' polynomial where we take the integers
$k_1,\ldots ,k_b$ mod $p$ as well as the sum of the exponents 
$k_1v_1+\ldots +k_bv_b$. Moreover the latter polynomial will vanish on
evaluation at $(k_1,\ldots ,k_b)$ if and only if it does so at
$(\lambda k_1,\ldots ,\lambda k_b)$ if $\lambda$ is invertible mod $p$.
Therefore for a suitable number of small primes $p$ we ``box'' $N$
mod $p$ to obtain $\overline{N}_p$ and we evaluate this under the 
$(p^b-1)/(p-1)$ equivalence classes of homomorphisms. We keep a record
of which of these homomorphisms make $\overline{N}_p$ vanish (we can choose
just to look mod 0, but it is better to keep further records 
as to which $n$ make
$\overline{N}_p$ equal to 0 mod $n$). If for a particular $p$ we find no
homomorphism makes $\overline{N}_p$ vanish then we can break immediately
and conclude that no Alexander polynomial of $G$ is zero. Otherwise we can
construct possible values for $(k_1,\ldots ,k_b)$ and try these out
directly by evaluating $f$ at $\chi$.

There is another way of eliminating $G$ without even calculating the
Alexander polynomial which we now summarise: for any surjective
homomorphism $\chi$ we set $K=$ ker $\chi$ so that we have the
cyclic covers $K\leq G_m\leq G$ with $[G:G_m]=m$. Then if $\Delta_{G,\chi}$
is 0 (mod 0 or mod $p$) we must have that the rank (over $\q$ or over
$\z/p\z$) of the abelianisation $G_m/G_m'$ is at least $m+1$ (as we have
already seen for $m=1$). As for adapting this to when $b\geq 2$, upon
taking the kernel $N$ of the natural map from $G$ to its free abelianisation
$G/N=\z^b$ we have for any prime $q$ that
$NG^q\leq G_q\leq G$ with $[G_q:NG^q]=q^{b-1}$ as 
$G/NG^q$ is isomorphic to $(\z/q\z)^b$. If we know the abelianisation of
$NG^q$ then we can use this information in the following way: if
$\beta_1(NG^q)\leq q$ then there is no $\chi$ with $\Delta_{G,\chi}=0$
mod 0. Moreover if $p$ is a prime not equal to $q$ 
such that there is a surjective
homomorphism from $NG^q$ to $(\z/p\z)^{q+1}$ then $G_q$ having $p$-rank
at least $q+1$ implies that $NG^q$ does too as the index is coprime to
$q$ (whereas if $p=q$ then we can only conclude that $NG^p$ has $p$-rank
at least $p-b+2$ by considering the vector space $G_p/(G_p)'(G_p)^p$).
Thus at each stage we have a possible value $n$ for $\Delta_{G,\chi}$ being
0 mod $n$, and on taking the next prime $q$ and calculating the
abelianisation of $NG^q$ we can alter $n$ according to the above,
breaking and rejecting $G$ if $n$ reaches 1 whereas if $n$ survives for
several primes we can be reasonably sure that there will exist $\chi$
with $\Delta_{G,\chi}$ 0 mod $n$. This is because if there is no prime
$p$ for which $\Delta_{G,\chi}$ is zero mod $p$ then there is a bound
$B$ such that for all $m$ and all primes $p$ the $p$-rank of the
abelianisation of $G_m$ is at most $B$.

In fact Magma does not have a command that, given a finite presentation
of $G$, directly finds $NG^q$ but it does give $G'G^q$, so we can use
that instead and alter the expected rank accordingly, or we can obtain
$NG^q$ from $G'G^q$ by creating the homomorphism from $G$ to $G/NG^q$ via
$G/G'G^q$ and taking its kernel, or if $G/G'=\z^b\times T$ where $T$ is the
torsion subgroup then we can just run our check using primes $q$ which do not
divide $|T|$ as then the above bounds for the 
ranks of the abelianisation of $NG^q$
and $G'G^q$ will be equal. However we remark that for, say, $q=11$ and
$b=4$ we will be requiring the abelianisation of a subgroup $NG^q$ of $G$
having index over 10,000 (or if we are using $G'G^q$ then this is already
the case when $G/G'$ is, say, 
$\z^2\times(\z/11\z)^2$) and this can take a while to
find, in which case we would be better off going straight for 
the minors anyway. We remark that in no case did we witness a
delay in calculating the Alexander matrix; instead it was evaluating large
determinants that could take some time. 

Finally we note that when we talk about the algorithm being fast or slow
then we are referring throughout to practical rather than theoretical
running time. Indeed if it is the case that the property of largeness
is not algorithmically solvable amongst finitely presented groups (as
seems to be the prevailing view) then the theoretical running time for
any partial algorithm must be an uncomputable function, as otherwise we
could wait until the predicted number of steps to prove largeness had
occurred, and then conclude that the inputted group was not large. 

\section{2 generator 1 relator presentations}

\subsection{Presentations in standard form}

The subclass of finitely presented groups which have a 1 relator
presentation has been much studied. If this presentation has at least
3 generators then it is well known that the group is large, so it is
only 2 generator 1 relator presentations that are in doubt. We shall
shortly mention the presentations of this form which are known not to yield
large groups but we presented theoretical results in \cite{meis} and
\cite{megg} suggesting that a 2 generator 1 relator presentation is very
often large, and here we will present strong experimental evidence.

Given any 2 generator 1 relator
presentation $\langle x,y|w(x,y)\rangle$, we do of course get the same group
if we take a conjugate of $w$ or $w^{-1}$. However there are many other
presentations defining the same group, so to avoid dealing with this we will
only consider presentations in what we call Magnus form. 
This is based on the fact
that there is an automorphism $\alpha$ of $F_2$, with let us say
$\alpha(x)=a$ and $\alpha(y)=t$,
such that $w(\alpha^{-1}(a),\alpha^{-1}(t))$ has exponent sum 0 in $t$
when written as a word $w'$ in $a$ and $t$. We then say that $w'$ is in
Magnus form with respect to $t$. This is a consequence of the fact that
the kernel of the natural map from Aut($F_2$) to Aut($F_2/F_2')\cong
GL(2,\z)$ is the group of inner automorphisms of $F_2$.

The elementary Nielsen moves on an ordered pair of group elements
$(g_1,g_2)\in G\times G$ are: swapping the pair, replacing either with its
inverse and replacing $g_1$ with $g_1g_2$ or $g_2$ with $g_2g_1$. These
moves, when regarded as elements of $Aut(F_2)$ by their effect on the
standard basis $(x,y)$, generate Aut($F_2$). We say that two pairs are
Nielsen equivalent if there is a finite sequence of elementary Nielsen
moves taking one to the other, so in $F_2$ the equivalence class of $(x,y)$
consists precisely of all generating pairs.

If the group $G$ is given by a presentation $\langle a,t|w(a,t)\rangle$
with $w$ cyclically reduced and in Magnus form with respect to $t$ then
we can keep $w$ in standard form by replacing $t$ with $ta^k$ for any
$k\in\z$ (or sending $t$ or $a$ to their respective inverses), but
$Out(F_2)$ being isomorphic to $GL(2,\z)$ implies that these are the only
automorphisms of $F_2$ we can make up to conjugation which preserve
Magnus form (at least if $\beta_1(G)=1$: if $\beta_1(G)=2$ then every
2 generator 1 relator presentation of $G$ is in Magnus form with respect
to both generators). However, as we shall see later, there exist
2 generator 1 relator groups with more than one Nielsen equivalence class
of generating pairs. 

Therefore given any 2 generator 1 relator presentation, we can assume $w$
is of the form
\begin{equation}t^{k_1}a^{l_1}\ldots t^{k_n}a^{l_n}  
\end{equation}
where $k_i,l_i\neq 0$ and $k_1+\ldots +k_n=0$ (excluding the words $a^l$).
Thus writing $a_i=t^iat^{-i}$ we have
\[w=a_{s_1}^{l_1}a_{s_2}^{l_2}\ldots a_{s_n}^{l_n}\qquad\mbox{ where }
s_1=k_1\mbox{ and } s_{i+1}=s_i+k_{i+1}.\]
We call $h=\,$max$(s_i)-$min$(s_i)$ the height of $w$ when in Magnus
form with respect to $t$ and $2n$ the syllable length of $w$.
Note that the moves above preserving Magnus
form also preserve the height and the syllable length, 
though not necessarily the word length, of $w$. 

The presentations which are known not to be large all fall into two
types: first if the syllable length is 4 with $k_1=-k_2=1$ then we have the
Baumslag-Solitar group $BS(l_1,-l_2)$ and it is well known that this is not 
large if and only if $l_1$ and $l_2$ are not coprime. 
The other type comes from \cite{megg} Theorem 4.3 
(based originally on a construction by Higman in \cite{hg}; see also
\cite{bmgg}) 
which states that if $g$ and $h$ are
conjugate elements of a group $G$ where the relation $h^kg^lh^{-k}=
g^{l\pm 1}$ holds in all finite images of $G$ then $g$ and $h$ must be
trivial in any finite image of $G$. In particular if 
$G=\langle a,t|b^ka^lb^{-k}=a^{l\pm 1}\rangle$ where $b$ is an element of
$F_2$ which is conjugate to $a$ then the presentation is in Magnus form
with respect to $t$ but all finite images are cyclic so $G$ is not large.
Taking $k=l=1$ and the plus sign with $b=tat^{-1}$, we obtain the famous
group first introduced by Baumslag in \cite{baum}. We shall refer to this
ubiquitous group as the Baumslag-Brunner-Gersten group $BBG$. Note that
further examples can be obtained by iterating this construction because
if $G=\langle a,t|w\rangle$ is a presentation where $a$ is trivial in
every finite image of $G$ then we can take $w$ and a conjugate of $w$
to form a new relator where $w$ is trivial in every finite image of
this new group, thus $a$ is too.

We also remark that these non-large groups (as well as some others)
typically have unusual
presentations not Nielsen equivalent to the well known ones. This
observation dates back to \cite{mccpe} and is based on the following
trick: if $G=\langle x,y|w(x,y)\rangle$ is such that we can write
$w$ in the form $u(x,y)^k=y$ then by introducing the letter $z=u(x,y)$
we have that $G$ is also $\langle x,z|z=u(x,z^k)\rangle$ which for
$k\neq 0,\pm 1$ is generally not Nielsen equivalent. For instance
taking the presentation $\langle x,y|(xy^rx^{-1}y^{\pm r})^k=y\rangle$
we get $\langle x,z|xz^{kr}x^{-1}z^{\pm kr}=z\rangle$ so we have an
alternative presentation for $BS(kr,\mp kr+1)$. These are in Magnus
form with respect to $x$ but have longer syllable length than the usual
presentation so cannot be Nielsen equivalent. Also it was shown in
\cite{brun} that for any $s\geq 1$
\[\langle a,t|(ta^{2^s}t^{-1})a(ta^{2^s}t^{-1})^{-1}=a^2\rangle\]
is isomorphic to $BBG$ by putting $b=a^2$ so that 
$a=(tb^{2^{s-1}}t^{-1})^{-1}b(tb^{2^{s-1}}t^{-1})$, 
thus giving the same relation
with $b$ and $s-1$ in place of $a$ and $s$.

This is also one of a family of examples in \cite{brunon}:
\[D(k,l,m)=\langle a,t|(ta^kt^{-1})a^l(ta^kt^{-1})^{-1}=a^m\rangle\]
so that $BBG\cong D(2^s,1,2)$ for $s\geq 0$. Note that $D(k,l,m)$ is large
by the Alexander polynomial if $|l-m|\neq 1$ and has only finite cyclic
quotients otherwise by the above. Also for the sake of the tables later,
we point out that not only is $D(k,l,m)$ isomorphic to $D(-k,-l,-m)$ by
sending $a$ to $a^{-1}$ but also to $D(-k,l,m)$ by further inverting
both sides of the relation. Therefore we can assume that $k$ and $l$ are
always positive. 

\subsection{Height 1 presentations}

Let us suppose that  
the height of a presentation in the form (1) is 1, so that $n$ is even and
(by sending $t$ to $t^{-1}$ if necessary) 
$k_i$ is 1 for $i$ odd and $-1$ for $i$
even. Here we do not need anything as involved as Section 2 to determine
largeness because \cite{megg} Corollary 4.2 tells us that $G$ is large
if and only if there is $H\leq_f G$ with $d(H/H')\geq 3$, where $d$ is the
minimum number of generators for a group. Therefore we merely need to
compute for each $i\geq 2$ the index $i$ subgroups of $G$ and their
respective abelianisations, breaking as soon as one is found needing at
least 3 generators. When we have done this for reasonably high $i$, we
can look at which presentations are left and ask whether we already
recognise them as groups known not to be large.
To do this efficiently in such a way that we
are not repeatedly taking the same group with many different presentations
but without obsessively demanding only one presentation per group we adopt
the following: first note that for height 1 presentations the move
$t\mapsto ta^k$ does not change (1). Then by taking cyclic permutations
and sending $a$ to $a^{-1}$ if necessary, we can assume that $l_1>0$ and
$l_1\geq |l_i|$ for all $i$. We can further arrange it by reversing the
word that $l_2>l_n$, or if equal that $l_3>l_{n-1}$ and so on. We then
choose an upper bound for the total word length we will consider. In our
case it was 14, not so much because of lack of computer power but by
human limitation on the numbers of left over presentations to be
inspected. Then for syllable length $2n$ (starting at $n=4$ because 
$n=2$ only
yields Baumslag-Solitar groups) and a fixed $l_1$ (starting at 1) we have an
upper bound $b$ on the moduli of the other $l_i$ (so that $b$ will equal
$l_1$, or less if that value always makes the word too long), so we
take all presentations where the values of $(l_2,\ldots ,l_n)$ are
counted (ignoring zeros)
from $(-b,\ldots ,-b)$ to $(b,\ldots ,b)$ and we input them if
the word length is at most 14. We then choose a bound for the index of the
subgroups we examine which is as high as possible without long delays in
finding all subgroups of this index: as some groups will be proved large
well before this we went up to index 12, which with these presentations
usually only took a few seconds.

This allows us to draw up an initial list of what might fail to be large.
The policy adopted from then on was as follows: first find the
Alexander polynomials and try to show that the group is $BS(m,n)$ or
$D(k,l,m)$ by use of the transformations above. If this failed then it
is also a consequence of \cite{megg} Section 4 that if $G$ is given by
a height 1 presentation but is not large then the finite residual $R_G$ is
equal to $G''$. If $G$ has the same Alexander polynomial as $B=BS(m,n)$
then, as $R_B=B''$ too with $G/G''$ isomorphic to $B/B''$, any finite index
subgroup of $B/B''$ is also one of $G/G''$ and hence corresponds under the
inverse image to a unique one of $G$ and of $B$ too. Moreover these subgroups
will have the same abelianisation throughout, as if $H\leq_f G$ then
$H''\geq R_H=R_G=G''\geq H''$ so $H''=G''$, meaning that the abelianisation
of the corresponding subgroup $H/G''$ of $G/G''$ is $(H/G'')/(H'G''/G'')
\cong H/H'G''=H/H'$ which is equal to that of $H$. Also the same holds
for finite index subgroups of $B$, so we compare the abelianisations of 
subgroups of $G$ and of $B$ up to index 12, and even though it may be that
all such subgroups $H$ of $G$ have $d(H/H')=2$, the torsion of $H/H'$ might
have higher order than the torsion of the abelianisation of the 
equivalent subgroup of $B$. This worked for most
of the remaining presentations, with for some reason index 7 subgroups
very often successful.

The upshot was that this process left only four presentations unknown as
to whether they were large or not. Table 1 lists all height 1 presentations
of length up to 14 giving rise to non large groups, in order of syllable
length (starting at 4) and then word length, along with these four.
Two we could resolve: they were (using
capital letters for inverses)
\[G_1=\langle a,t|ta^3Ta^2ta^3TA\rangle\mbox{ and }
  G_2=\langle a,t|ta^3Tata^3TA^2\rangle.\]
Their Alexander polynomials are the same as for $BS(1,\mp 6)$ but they
cannot be these groups as $BS(1,m)$ has only one Nielsen equivalence class
of generators. This is because the latter groups are soluble, whereas a
2 generator 1 relator presentation in Magnus form of height $h$ implies
that the group must contain a free subgroup of rank $h$. However they are also
ascending HNN extensions (with base $\z$) and this means that a 2 generator
1 relator presentation of $BS(1,m)$ in Magnus form must have a unique
maximum, but the word can only be of height 1. Therefore
working on the principle that these groups are large (because if not then
we have no idea what groups they are or how we would prove them not to be
large), we looked for non-abelian finite simple images. The computer found
that $G_1$ mapped onto the unitary group $U(3,3)$ of order 6048 and $G_2$
has the sporadic simple group $J_2$ of order 604800 as a finite image.
This means that we do not have $R_{G_i}=G_i''$ so $G_1$ and $G_2$ are large
by \cite{megg} Theorem 4.1.

However that still leaves two groups whose Alexander polynomial is 1,
so largeness is determined by the answer to the following:
\begin{qu} For the groups with presentations
\[
\langle a,t|ta^2TatATatATA\rangle\mbox{ and }
\langle a,t|ta^2TatATAtATA\rangle,\]
is it true that $a$ is trivial in every finite image?
\end{qu}

Consequently we have established that the vast majority of 2 generator
1 relator height 1 presentations with short word length define large
groups. However the only way we have determined that some groups are not
large was to show them to be isomorphic to groups already known to have
this property and there seems to be a very limited number of methods
which are able to prove that presentations of this type are not large.

\subsection{General height presentations}

We now consider 2 generator 1 relator presentations $\langle a,t|w(a,t)
\rangle$ in standard form with respect to $t$ but with height greater
than 1. Here we need to implement using Magma the algorithm described
in Section 2. Before we embark on running it, we note that not all
presentations of this type define large groups. The only example we
know of already in the literature is
\[\langle a,t|(t^kat^{-k})a(t^kat^{-k})^{-1}=a^2\rangle\]
in \cite{mol?} which has height $k$. However as $t^k$ can be replaced
here by any word in $F_2$, it is clear that there are a lot more such
presentations. Nevertheless it seems remarkable that we will not
encounter any non-large presentations in this subsection of word length
at most 12, given that the above example provides one of length 13.

We will have to work much harder than in the height 1 case. Although
word length at most 14 and subgroup index up to 12 was the starting
point for this work, too many presentations of length 13 were left
over to be dealt with by hand. These would then need to be checked further
by increasing the index, but finding the subgroups of index 15 can already
take 
a fair while with such a presentation. The other problem is that
rather than just requiring the abelianisation of these finite index
subgroups, we need to rewrite to get a finite presentation. Although this
was not generally a concern, it began to fail to complete
in a suitable time
on a few of the length 13 and 14 cases. However an important point which
makes these presentations tractable is that although a finite index
subgroup of a 2 generator 1 relator group will not usually have such
a presentation itself, it will still have a deficiency 1 presentation
so that for each subgroup only one minor need be calculated.

This time we assumed that, given a word $w$ in the form of (1), the
exponents of the letter $t$ were such that $k_1\geq |k_i|$ for $i\leq 2\leq n$.
We can assume that $k_1\geq 2$ because if all $k_i=\pm 1$ but $w$
does not have height 1 then there must be a subword (where $w$ is regarded
as a cyclic word) of the form $ta^lt$ (or $t^{-1}a^lt^{-1}$).
We can then apply the transformation $t\mapsto ta^{-l}$ 
(or $t\mapsto ta^l$) to obtain a subword of the
form $t^{\pm 2}$ whilst still keeping $w$ in Magnus form, but this could
increase the word length of $w$. To get round this, observe that subwords
of the form $ta^lt^{-1}$ or $t^{-1}a^lt$ are unchanged by this process,
whereas there is a pairing in which every subword of the form $ta^lt$ in $w$ 
is placed with one of the form $t^{-1}a^{l'}t^{-1}$. 
If we take the $l$ of smallest modulus appearing in
any subwords of $w$ in either of these latter two forms 
then the transformation $t\mapsto
ta^{-l}$ in the first case and $t\mapsto ta^{l'}$ in the second case
will obtain a subword of the form $t^2$. Let us assume that it is
$ta^lt$ with $l>0$. If the word length of $w$ is increased by this
substitution then it must be that the pairs of subwords of the form
$ta^mt$, $t^{-1}a^nt^{-1}$ for $m>0$ and $n<0$ (where the combined length
increases under $t\mapsto ta^l$) outnumber the pairs where $m<0$ and $n>0$
(where we decrease). Consequently the move $t\mapsto ta^{-l}$ will reverse 
this situation and we will decrease the total word length of $w$. Although
we may still not have a proper power of $t$ in $w$, we can repeat this process
until we do.

For a given $n$ we then fixed a $t$-shape, which we define to be the
vector $(k_1,\ldots ,k_n)$, and we had a bound $b$ such that if any
exponent $l_i$ of the letter $a$ had modulus greater than $b$ then the word
was too long. We then inputted all $a$-shapes $(l_1,\ldots ,l_n)$ from
$(-b,\ldots ,-b)$ to $(b,\ldots ,b)$ if the resulting word length was
at most 12. Note that because we can send $a$ to $a^{-1}$ in a presentation,
each group has been inputted at least twice. This is no bad thing because
it acts as a check, and given the method we were using it could happen
that one presentation was proved large whereas the other was missed.
This anomaly can only come about on subgroups with first Betti
number greater than 2 because we did not initially
bother incorporating
 a full working algorithm for subgroups $H\leq_f G$ with $\beta_1(H)\geq 3$ 
in the run owing to their
comparative rarity. In fact the approach adopted when working through
the lists of presentations was that of Section 2
for first Betti numbers 1 and 2, whereas for $\beta_1(H)\geq 4$ we only
looked at the minors $N(1,\ldots ,1,t_j,1,\ldots ,1)$ evaluated at just
one variable for each $j$ between 1 and $\beta_1(H)$.
For $\beta_1(H)=3$ we at least considered the double
variable polynomials $N(x,y,1)$, $N(x,1,z)$ and $N(1,y,z)$ and fed
these to the $\beta_1(H)=2$ routine. 

The initial results when run up to index 12 were that there was one
remaining presentation not proved large of length 9, 4 of length 10,
less than 30 of length 11 and none of length 12. We then ran the program
on these remaining presentations on subgroups of index 13, 14 and 15.
It would take about a minute to obtain all the subgroups of index 15 and
the Alexander polynomial checks would then be instant. In one case at
index 15 it failed to rewrite for a subgroup presentation, but we then
ran the abelianised version at the end of Section 2 which eliminated
this case. 

We were then only left with the one length 9 presentation and 15 length 11
presentations as given in Table 2. We now need to
find other ways of showing that these groups are large. We can use the
Magma command LowIndexNormalSubgroups which will be able to find normal
subgroups up to a higher index than that for arbitrary subgroups, but there
will be less of them. The main tool we use is a form of bootstrapping:
pick a low index subgroup $H$ with decent abelianisation, by which we usually
mean $\beta_1(H)$ at least 2 or (though preferably and) $d(H/H')\geq 3$. Then
rewrite to get a finite presentation for $H$ and regard this as our input.
Although the presentation will get longer so that we will be unable to find
all subgroups up to index 15 again, a subgroup $L$ of $H$ will have index
$[G:H][H:L]$ in $G$ so we can go considerably higher by being selective. We
can even repeat the process until the calculations become too long.
At this point we also developed as a separate
routine the algorithm in Section 2 for $\beta_1(H)$ at least 3 and this
was able to prove largeness when applied to an appropriate subgroup $H$
of the very few groups left over.
 
Let us illustrate using the solitary presentation of length
9 which was left over, as this turns out to be an interesting group that
has already appeared in the literature. The presentation is
$\langle a,t|t^3at^{-2}a^{-1}t^{-1}a^{-1}\rangle$ and the group $\Gamma$
it defines is 
free by cyclic as its homomorphism onto $\z$ has kernel the free group
$F_3$ of rank 3. This can be seen by setting $b=tat^{-1}$ and $c=tbt^{-1}$
so that we have an alternative presentation of the form
\[\Gamma=\langle a,b,c,t| tat^{-1}=b, tbt^{-1}=c, tct^{-1}=ab\rangle\]
which happily is the subject of the paper \cite{sta} by Stallings
where it is shown that the corresponding automorphism $\phi$ of $F_3$
is not topologically realisable as a homeomorphism of a compact surface
with boundary. This was followed up in the paper \cite{gerst} where
conditions were given to ensure that all positive
powers of an automorphism
$\alpha$ of the free group $F_n$ are irreducible, 
where $\alpha$ is said
to be reducible if there exist proper non-trivial free factors $R_1,
\ldots ,R_k$ of $F_n$ such that the conjugacy classes of $R_1,\ldots
,R_k$ are permuted transitively by $\alpha$ (see \cite{bh}). In particular
it was shown that all positive powers of $\phi$ are irreducible, 
and furthermore that for $k\geq 1$ no power $\phi^k$ can fix a 
non-trivial word, even up to conjugacy. 

As for largeness of groups $G=F_n\rtimes_\alpha \z$ for $n
\geq 2$, this is true by \cite{meis} if $G$ contains $\z\times\z$, which
is equivalent to there being a non-trivial element $w$ in $F_n$ and $k\geq 1$
such that $w$ is sent by $\alpha^k$ to a conjugate of itself. It is also
equivalent to a group of this form not being word hyperbolic. Thus our
group $\Gamma$ above is word hyperbolic. However we can have cases
where $G$ is not word hyperbolic but all powers of $\alpha$ are
irreducible, such as if $\alpha$ is an automorphism of 
$F_2=\langle x,y\rangle$ then $\alpha^2$ will always fix the
conjugacy class of $xyx^{-1}y^{-1}$ but $\alpha^k$ need not fix a generator
of $F_2$ for any $k$. We can also have $G$ being word hyperbolic but
$\alpha$ reducible; indeed in \cite{megg} the first example of a large
word hyperbolic group of the form $F_n\rtimes_\alpha \z$ was given by
putting together two copies of Stallings' automorphism $\phi$. However
here we can give the first example of something rather stronger.
\begin{thm} 
The word hyperbolic group $\Gamma=F_3\rtimes_\phi\z$ above is large
even though $\phi^k$ is irreducible for all $k\geq 1$.
\end{thm}
\begin{proof}
From the above we just need to establish the largeness of $\Gamma$ and this
is done purely computationally using the Magma program described in Section
2. On looking for the abelianisations of the low index subgroups of $\Gamma$
(these are strictly speaking low index subgroups up to conjugacy, which
helps us as we avoid duplication of non normal subgroups),
the only promising avenue is the first subgroup $H$ of index 7, with
abelianisation $C_2\times C_2\times C_2\times\z$. We obtain a presentation
for $H$ by rewriting using the presentation for $\Gamma$ and then use the
command LowIndexNormalSubgroups on $H$ up to index 8. We find 191 of these
subgroups, with number 121 having index exactly 8 and abelianisation
$L/L'=(C_2)^2\times(C_4)^3\times\z^4$. On again rewriting, this time for
$L$ in $H$ (thus obtaining 9 generators and 8 relators for an index 56
subgroup of $\Gamma$), we then apply our
routine to the resulting presentation of $L$. The Alexander polynomial
$\Delta_L(t,u,v,w)$ occupies 10 sides of printout and would surely not be
calculable without the computer but it also determines that
$\Delta_L(1,1,v,1)$ is equal to
\[48v^{27}-256v^{26}+576v^{25}-768v^{24}+800v^{23}-768v^{22}
+576v^{21}-256v^{20}+48v^{19}\]
which is of course 0 mod 2 so we have established that some finite index
subgroup of $L$, and hence a subgroup of $\Gamma$ with index a multiple of
56, surjects onto a non abelian free group.
\end{proof}

Note that although running through all the subgroups calculated in the lists
above until the right one is found could have running time of an hour, once
we know where to look we can check our answer in a minute or two. 
We give a summary of arguments which established largeness for the other 
15 presentations in the Appendix. Consequently we have:
\begin{thm}
If $G$ is given by a 2 generator 1 relator presentation 
$\langle x,y|r\rangle$ in Magnus form and $r$ is cyclically reduced
with length at
most 12 then $G$ is not large if and only if either the presentation appears
in Table 1 (up to cyclic permutation of $r$ or $r^{-1}$ and replacing
either generator by its inverse)
or the presentation is of the form 
$\langle a,t|ta^pt^{-1}=a^q\rangle$ where $p$ and $q$ are coprime.
\end{thm}
This also gives the full picture for the virtual first Betti number of
these groups:
\begin{co}
If $G$ is as in Theorem 3.3 then $G$ has infinite virtual first Betti
number unless the presentation appears
in Table 1 (up to cyclic permutation of $r$ or $r^{-1}$ and replacing
either generator by its inverse) in which case the virtual first Betti
number is 1, or the presentation is of the form 
$\langle a,t|ta^pt^{-1}=a^q\rangle$ where $p$ and $q$ are coprime, in which
case the virtual first Betti number is 2 for $p,q=\pm 1$ and 1 otherwise.
\end{co}
\begin{co}
If $G$ is as in Theorem 3.3 
with $r$ having length at most 12 and height at least
2 then $G$ is large.
\end{co}
As noted at the start of this subsection, this is not true once we move to
length 13.

\subsection{First Betti number equal to two}
It may have been noticed in this section that if the relator $r$ is in the
commutator subgroup $F_2'$ (which is equivalent to the group 
$G=\langle x,y|r\rangle$ having $\beta_1(G)=2$) then no examples have yet
been found where $G$ is not large (except for $r=[x,y]=xyx^{-1}y^{-1}$).
Although it is possible for $G$ not to be large in this case, such as the
example $\langle x,y|[y^{-1},x][x,y][y^{-1},x]^{-1}=[x,y]^2\rangle$ in
\cite{bmgg} Corollary 2, there is theoretical evidence that $G$ is often
large; the strongest result in this area is \cite{meis} Theorem 3.6 stating
that if $r$ is actually a commutator and there is $H\leq_fG$ with
$d(H/H')\geq 3$ then $G$ is large (and if the
condition on $H$ is removed then the only known counterexample is 
$\z\times\z$).

In this subsection we describe calculations which show that if $G=\langle
x,y|r\rangle$ for $r\in (F_2)'$ with word length at most 16 
then either $G=\z\times\z$ (which is well known only to happen if
$r=[x,y]$ when reduced and cyclically reduced, up to a cyclic permutation
of $r$ or of $r^{-1}$, so $r$ always has length 4) or $G$ is large. The
example given above suggests this would not be true for word length 20;
in fact writing this word out reveals that it has length 18. Whilst it
might not be remarkable that no counterexamples exist that have length
shorter than 18, we also show with calculations on a substantial number
of cases the somewhat more surprising result that this is the only
counterexample amongst words of length 18. 

Given $r\in F_2'$, it must have even length. From 3.3 we are fine for
length at most 12, and in fact for length 14 too: we check that no
relator $r$ left over in the lists for length 14 is in $F_2'$. This
covers all cases except the 3 $t$-shapes where the program failed to 
complete: these were $t^2a^{l_1}t^{-2}a^{l_2}ta^{l_3}t^{-1}a^{l_4}$,
$t^2a^{l_1}t^{-1}a^{l_2}ta^{l_3}t^{-1}a^{l_4}t^{-1}a^{l_5}$ and
$t^2a^{l_1}t^{-1}a^{l_2}t^{-1}a^{l_3}$. On running the program again
with the condition that the exponent sum of $a$ is zero, we very quickly
find that all such presentations of length 14 give large groups.

Moreover we have by the combined results of
\cite{edph} and \cite{meis} that $G$ is always large (or $\z\times\z$)
if $r\in F_2'$ and has syllable length 4 or 6. Therefore when we move on
to length 16 we need only consider
words with 8 syllables or above. We have that $r$ is automatically in
Magnus form with respect to both $a$ and $t$, and remains so under
any Nielsen transformation although the word length can change. We
therefore assume that $k_1>0$ and has largest modulus amongst both the
$k_i$ and the $l_i$ (else we can swap $a$ and $t$). We then fix the $k_i$
whilst counting through the various $l_i$ as in 3.3 but subject to
$l_1+\ldots +l_n=0$ and word length exactly 16. Note also that we do not
need to consider $k_1=1$, which is where the syllable length is also 16
and $|k_i|=|l_i|=1$ for each $i$. This is because if the appearances of
$k_i$ do not alternate in sign then we can use the substitution as
described in Section 3.3 to either reduce the length of $w$ or get
$k_1\geq 2$. However if the signs of the $k_i$s do alternate then we observe
that the Alexander polynomial with respect to $t$ must be 0 modulo 2.

However here we begin to encounter the problem that, whilst all subgroups
of index up to 12 can usually be found, it takes too long to rewrite for
presentations of certain subgroups of quite low index and our run becomes
stuck. We deal with this when it occurs by breaking and performing the
whole run of cases again up to a lower index (often index 6, 7 or 8).
We might then have a few presentations left over which have not been
proved large, but as before we can draw up a list and use ad hoc
arguments to deal with these cases. 
This list is given in Table 3, and in fact the only arguments needed to
conclude largeness for the leftover words are to take the Alexander polynomial
of a particular finite index subgroup with promising abelianisation
(where in the table we give the index, number within that index, and
abelianisation of this subgroup) or to
make a Nielsen move which converts the word into one that has already been
proved large (where we give this move, along with 
the number of the new presentation
if it is in this table, otherwise the new presentation was proved large
in the run without a problem).
There are also a tiny number of presentations which complete 
the rewriting process on all subgroups up to index 12 without being proved 
large, but the same arguments will work on these too.

We then move up to length 18 and proceed in the same way. Once again we list 
in Table 3 the presentations where we failed to rewrite for a subgroup
along with arguments that establish largeness; once again these work in the
same way as for length 16. However the main consumption of time is taken
up by typing up each entry from the long list of possible $t$-shapes and
waiting for the program to finish (or occasionally get stuck). We invoked
one short cut here: suppose the $t$-shape is such that the only possibilities
for the indices $l_i$ of $a$ are all $|l_i|=1$. 
If these indices all alternate in sign then we check this one 
particular case using the $d\geq 3$ test as in Section 3.2. Otherwise 
we swap $t$ and $a$ and then apply a transformation as before to ensure that
the exponents of the new letter
$t$ are not all $\pm 1$ (or the word may become
shorter under this process, in which case we are done). This is fine unless
we now find that all the exponents of the new letter
$a$ have modulus 1, which
would put us back where we started, but this can only happen by shortening
the word. Note also that we cannot have $|k_i|=|l_i|=1$ for all $i$ as
this implies that the word length is divisible by 4.

We find that all cases are proved large apart from when the $t$-shape is
$(2,-1,-1,2,-1,1,-1,-1)$ and the exponents of $a$ have modulus 1 and
alternate in sign, which is our one exception given earlier. Amusingly
we also get a ``near miss'' immediately before with $t$-shape
$(2,-1,-1,2,1,-1,-1,-1)$ where all subgroups up to
index 9 have abelianisation $\z^2$, but with lots of index 10 subgroups
having first Betti number equal to 4 so this group is large after all.

Thus we have a result that is definitive for words in the commutator 
subgroup with length at most 18.
\begin{thm} Suppose $G$ is given by a 2 generator 1 relator presentation
$\langle x,y|r\rangle$ with $r$ cyclically reduced and in the commutator
subgroup of $F_2$, with the word length of $r$ at most 18. Then
either $r$ has word length 4 and $G$ is isomorphic to $\z\times\z$ or $r$
has word length greater than 4 but less than 18, in which case $G$ is large,
or $r$ has word length 18 and $G$ is large except for
$r=[y^{-1},x][x,y][y^{-1},x]^{-1}[x,y]^{-2}$, up to cyclic permutation of
$r$ or $r^{-1}$ and replacing either generator by its inverse,
whereupon $G$ contains a non
abelian free group but all finite images of $G$ are abelian and all finite
index subgroups of $G$ have abelianisation $\z\times\z$.
\end{thm}

\section{Closed hyperbolic 3-manifolds}

The other examples on which we tried out our program were the
fundamental groups of closed orientable hyperbolic 3-manifolds (from
hereon we refer to these as 3-manifold groups) and again we met with
some success. It is an open question as to whether every 3-manifold
group is large. The focus of our work was the closed census of hyperbolic
3-manifolds \cite{cen} that accompanies the program Snappea. This is
available as a database in Magma of 11,126 entries, each of which contains
a presentation of the corresponding 3-manifold group, along with the
manifold's name and volume (which provides an ordering for the database).
A 3-manifold group always has a presentation of deficiency 0 (and never
has a presentation of strictly positive deficiency). In the census we
find that most presentations are 2 generator 2 relator with a fair number
of 3 generator 3 relator presentations, mainly concentrated amongst the
higher volume manifolds. 

Consequently we have no guarantee that the group has a finite index subgroup
with positive first Betti number. As our program merely ignores finite index
subgroups which have no homomorphisms onto $\z$, there would be no point
in entering a group if no such subgroup was already known or easily found.
Fortunately this has already been covered in the paper \cite{dnth} by
N.\,Dunfield and W.\,Thurston in which it was proved that all 3-manifold
groups in this census have positive virtual first Betti number. However
this was a major computational undertaking involving about a year of CPU
time. The final group to be completed was found to have this property
by running Magma to obtain its subgroups of index 14 and then examining
the abelianisations. This took two days, which was about the time we had
available to look at the whole census. Moreover the highest index of a
subgroup that provided positive virtual first Betti number of a group
was 515,100 which (although this may not be the minimal index
for this group) is far
too high to expect the computer to provide a presentation to check for
largeness. (There is room for considerable improvement however. When
Magma calculates the abelianisation of a finite index subgroup obtained
from a given finitely presented group, it does not work out a full
presentation for the subgroup but abelianises the relations as it goes along.
When one works out the Alexander matrix of a finitely presented group
$G$, one is really calculating a presentation matrix 
for $N/N'$ when considered as a
$\z [t_1^{\pm 1},\ldots ,t_b^{\pm 1}]$-module, 
where $\beta_1(G)=b$ and $G/N=\z^b$,
so it should be possible to calculate this directly too. Moreover the paper
\cite{dnth} exploited the fact that one only wants to know if the 
abelianisation of a subgroup is infinite, rather than needing explicitly
to find the rank and torsion, and some representation theory of finite
groups was utilised at this point.)

Therefore the approach we adopted was to look for largeness amongst groups
which have a very low index subgroup with infinite abelianisation. Trying
this out with index at most 5, we find in less than two minutes (with the
code given at the end) that there are 2856 groups with this property, 
which at just 
over a quarter is a decent proportion. The choice of 5 was made by wanting
a bound that ensured we were covering a large enough sample of the census
in our program, but also because we ran our largeness routine up to index
10 on all of these groups. The reason for going up to double the initial
choice of
index is that, because a group with a homomorphism onto $\z$ has subgroups
of all finite indices with the same property, we know we will have
more than one subgroup of the original group with positive Betti number
that can be checked for largeness. (In fact we started with the 4948
groups with such a subgroup of at most index 6 and then tried finding
subgroups up to index 12: this was working well on the earlier groups
but became too slow when dealing with the 3 generator 3 relator presentations.
Note also that the proportion of groups having a low index subgroup with
positive first Betti number is higher here than in \cite{dnth} because
there other methods were first used to find such subgroups, rather than
enumeration of all low index subgroups.)

We then ran the routine described in Section 2 on this list of 2856
3-manifold groups, checking subgroups up to index 10. The results were
encouraging. We ran the program in batches of a few hundred at a time and
only 130 failed to be proved large by this method. Some
3-manifold groups were dealt with very quickly although several towards
the end could take quite a time; the longest wait was over an hour in
finding the subgroups of index 10 for number 10017. (There were three
cases where we gave up on finding the index 10 subgroups after waiting for
over an hour: numbers 10540, 10671 and 10922.)

For a few of the very early
3-manifold groups with more tractable presentations, we increased the index
for finding subgroups up to 12 or 14. This produced 4 other large groups
in our sample, although it worked for considerably more groups where the
earliest subgroup with positive first Betti number appears at index 6.
We then used the
data in Snap \cite{sn} to see if any 3-manifold groups in the leftover list
were arithmetic. This is because any arithmetic 3-manifold group in the
census is large, by combining \cite{lklr} Theorem 6.1 (an arithmetic
3-manifold group with a finite index subgroup having first Betti number
at least 4 is large), \cite{clr} or \cite{venk} (an arithmetic 3-manifold
group with positive virtual first Betti number has infinite virtual first
Betti number), and the Dunfield-Thurston paper (every census 3-manifold
has positive virtual first Betti number). This provided 8 more examples,
marked by ``Ar'' in Table 4,
although they all came from early on in the census. Combining the data,
we were left with 118 3-manifold groups left over that were not proved large.

Amongst these 118 were 2 groups which themselves have positive first Betti
number, out of 132 in the database. We pushed up the index to 11 or 12
here and selected particular subgroups to try. This established largeness
for the remaining pair (number 3552 which has volume 4.7874 and is called
v1721(1,4), and number 3763 which has volume 4.8511 and is s828(-4,3)).
Moreover we notice that all of the 305 further 3-manifolds in the census
having a double cover with positive first Betti number have been proved 
large, giving
\begin{thm}
If $G$ is the fundamental group of a closed hyperbolic 3-manifold in the
census and $\beta_1(G)\geq 1$ or $G$ has a double cover with positive
first Betti number then $G$ is large.
\end{thm}

The results are listed in Table 4. As we did not want to list all 2740
of the census 3-manifolds which have been proved large, the tables give
the 130 cases where the 3-manifold group had a subgroup of index at most 5
with positive Betti number but where largeness was not established by
this method up to index 10, with the 14 groups proved large by 
alternative means marked in the table with a tick.
Therefore in order to utilise the information for the $n$th manifold
in the database, one needs to first run a Magma program along the lines
of:\\
\hfill\\
\texttt{
> M:=ManifoldDatabase();\\
> G:=Manifold(M,n)`Group;\\
> for j:=1 to 5 do\\
> j; L:=LowIndexSubgroups(G,<j,j>);\\
> for k:=1 to \#L do\\
> A:=AQInvariants(L[k]);\\
> if \#A gt 0 and  A[\#A] eq 0 then "yes!"; break j; end if;\\
> end for;\\
> end for;
}\\
\hfill\\
If the answer yes is obtained then the 3-manifold group is large if it
does not feature in Table 4, or if a tick is present.

There was an initial concern with our
method. A finite index subgroup of a 3-manifold group always has
a presentation of deficiency zero because it is itself a 3-manifold group,
being the fundamental group of a finite cover of a 3-manifold. However
the computer does not know this and often on rewriting we obtained
presentations of strictly negative deficiency ($-4$ was the lowest value
observed, with $-1$ appearing quite often and $-2$ or $-3$ cropping up
occasionally). The approach we first
adopted during the runs was to ignore these subgroups, even though
they might prove largeness, because we feared the running time might be 
substantially
lengthened owing to the larger number of minors that need to be calculated.
This still managed to prove all but about 200 groups were large,
but was overcautious in that when the program was extended to include
presentations of arbitrary deficiency, it then if anything ran more quickly
because lower index subgroups could establish largeness earlier on.

Based on the above evidence we feel that the vast majority of the
census 3-manifold groups are likely to be large, because we have
only scratched the surface of what might be done using this approach,
given sufficient time and computing resources. However as to whether
this provides evidence for 3-manifold groups being large in general,
we would need to have an idea of whether the census 3-manifolds
provide a ``typical'' sample.

\section*{Appendix: Tables}
\markboth{APPENDIX: TABLES}{APPENDIX: TABLES}
{\bf Largeness in Table 2}\\
Here we briefly describe the arguments that allow us to conclude
that all presentations in Table 2 give rise to large groups.\\
\hfill\\
{\bf \#1}: See Theorem 3.2.\\
{\bf \#2}: Take the fourth subgroup of index 13 with abelianisation
$(C_3)^4\times\z$. Then the third index 3 subgroup of this with
abelianisation $(C_3)^4\times C_9\times\z^3$ has multivariable
Alexander polynomial $\Delta(x,y,1)=0$ mod
27.\\
{\bf \#3}: Take the third subgroup of index 15 with abelianisation
$(C_2)^3\times C_{22}\times\z$. Then find the normal subgroups up to
index 4; number 65 in this list is of index 4 and
has abelianisation $(C_2)^3\times C_4
\times C_{44}\times\z^3$. Moreover $\Delta(x,1,1)=0$ mod 2.\\
{\bf \#4}: Index 15 number 5 with abelianisation 
$(C_2)^3\times C_{14}\times\z$ has as its first index 2 subgroup a group
with abelianisation $C_2\times C_4\times C_{28}\times\z^2$ and
$\Delta(t,t)=0$ mod 2.\\
{\bf \#5}: Index 11 number 2 with abelianisation 
$C_6\times\z^2$ has as the seventh and last index 2 subgroup a group
with abelianisation $C_3\times\z^3$ and
$\Delta(t,t,t^{-1})=0$ mod 2.\\
{\bf \#6}: The index 8 number 2 subgroup $H$ with abelianisation 
$(C_3)^2\times C_{39}\times\z$ has as its first index 3 subgroup $L$
a group with abelianisation $C_3\times C_{18}\times C_{234}\times\z$ and
$\Delta(t)=0$ mod 2. Note that $\beta_1(L)=1$.\\
{\bf \#7}: The index 6 number 2 subgroup with
abelianisation $C_2\times C_6\times\z$ has 
the index 4 number 15 subgroup with abelianisation $C_2\times C_{4}\times 
C_{156}\times\z^2$ and $\Delta(1,y)=0$ mod 2.\\
{\bf \#8}: The index 7 number 2 subgroup with abelianisation $\z^2$ has 
as its first index 4 subgroup a group with
abelianisation $C_{340}\times\z^2$ and $\Delta(1,y)=0$ mod 17.\\
{\bf \#9}: The index 12 number 1 subgroup with abelianisation $C_5\times\z^2$ 
has as its first index 3 subgroup a group with
abelianisation $C_{10}\times\z^2$ and $\Delta(1,y)=0$ mod 2.\\
{\bf \#10}: After failing to make progress, it was noticed that this group
was isomorphic to \#12.\\
{\bf \#11}: This group is isomorphic to \#9.\\
{\bf \#12}: This was the hardest group to prove large in the whole paper.
The index 16 (finding these takes several minutes but no longer, which is
perhaps surprising for such a high index) 
number 7 subgroup with abelianisation 
$C_6\times\z^2$ 
has as its first index 2 subgroup a group with
abelianisation $C_3\times C_6\times\z^3$ and $\Delta(t^{-1},t^2,1)=0$ mod 2.\\
{\bf \#13}: The index 15 number 2 subgroup with abelianisation 
$(C_2)^2\times(C_4)^2\times
C_{36}\times\z$ has as its first index 2 subgroup a group with
abelianisation $(C_2)^2\times C_4\times C_8 \times C_{72}\times\z^2$ 
and $\Delta(x,1)=0$; this is actually equal to 0 as opposed to 0 modulo
something.\\
{\bf \#14}: This group is isomorphic to \#13.\\
{\bf \#15}: The index 12 number 2 subgroup with abelianisation 
$C_5\times\z^2$ has as its fourth index 3 subgroup a group with
abelianisation $C_{10}\times\z^2$ 
and $\Delta(1,y)=0$ mod 2.\\
{\bf \#16}: The first index 7 subgroup with abelianisation 
$(C_2)^3\times\z$ has as its second index 4 subgroup a group with
abelianisation $(C_2)^2\times C_4\times\z^3$ 
and $\Delta(t^2,1,t^{-1})=0$ mod 4.\\
\begin{table}
\caption{2 generator 1 relator height 1 presentations}
\centering
\begin{tabular}{|lrllc|}
\hline
Length&Number&Presentation&Description&Large\\
\hline
4.9 & 1 & $ta^2TataTa$   &$BS(2,-3)$ &\texttt x\\
    & 2 & $ta^2TatATA$   &$BBG$      &\texttt x\\
    & 3 & $ta^2TAtaTA$   &$BS(2,3)$  &\texttt x\\
4.11& 4 & $ta^2Ta^2ta^2Ta$&$BS(3,-4)$ &\texttt x\\
    & 5 & $ta^2Ta^2taTa^2$&$\cong \#4$ &\texttt x\\
&      6 & $ta^2Ta^2tATA^2$&$D(2,1,2)\cong BBG$&\texttt x\\
&      7 & $ta^2Ta^2tA^2TA$&$\cong \# 6$&\texttt x\\
&      8 & $ta^2TatA^2TA^2$&$\cong \# 6$&\texttt x\\
&      9 & $ta^2TAta^2TA^2$&$BS(3,4)$&\texttt x\\
&     10 & $ta^2TA^2taTA^2$&$\cong \# 9$&\texttt x\\
&     11 & $ta^3Tata^2Ta$&$BS(2,-5)$  &\texttt x\\
&     12 & $ta^3TatA^2TA$&$D(1,2,3)$ &\texttt x\\
&     13 & $ta^3TAta^2TA$&$BS(2,5)$ &\texttt x\\
4.13& 14 & $ta^3Ta^2ta^3TA$&      &\checkmark\\
&     15 & $ta^3Ta^2ta^2Ta^2$&$BS(4,-5)$&\texttt x\\
&     16 & $ta^3Ta^2tA^2TA^2$&$D(2,2,3)$&\texttt x\\
&     17 & $ta^3Ta^2tA^3TA$&$D(3,2,1)$&\texttt x\\
&     18 & $ta^3Tata^3TA^2$&&\checkmark\\
&     19 & $ta^3TatA^3TA^2$&$\cong \# 17$&\texttt x\\
&     20 & $ta^3TA^2ta^2TA^2$&$BS(4,5)$&\texttt x\\
&     21 & $ta^4Tata^3Ta$&$BS(2,-7)$&\texttt x\\
&     22 & $ta^4TatA^3TA$&$D(1,3,4)$&\texttt x\\
&     23 & $ta^4TAta^3TA$&$BS(2,7)$&\texttt x\\
6.13& 24 & $ta^2TataTataTa$&$BS(3,-4)$&\texttt x\\
    & 25 & $ta^2TatATatATA$&?&?\\
    & 26 & $ta^2TatATAtATA$&?&?\\
    & 27 & $ta^2TAtaTAtaTA$&$BS(3,4)$&\texttt x\\
6.14& 28 & $ta^2Tata^2TataTa$&$BS(3,-5)$&\texttt x\\
    & 29 & $ta^2TAta^2TAtaTA$&$BS(3,5)$&\texttt x\\
\hline
\end{tabular} 
\end{table}
\hfill\\ 
\begin{table}
\caption{2 generator 1 relator non height 1 presentations}
\centering
\begin{tabular}{|rrll|}
\hline
Length&Number&Presentation&Large\\
\hline
9 & 1 & $t^3aT^2ATA$&\checkmark\\
11& 2 & $t^4AT^3ATA$&\checkmark\\
  & 3 & $t^4AT^3aTA$&\checkmark\\
  & 4 & $t^4aT^3aTA$&\checkmark\\
  & 5 & $t^3aT^2A^2TA^2$&\checkmark\\
  & 6 & $t^3A^2T^2aTA^2$&\checkmark\\
&   7 & $t^3AT^2a^2TA^2$&\checkmark\\
&   8 & $t^3a^3T^2ATA$&\checkmark\\
&   9 & $t^3ATATA^2TA$&\checkmark\\
&  10 & $t^3aTaTA^2TA$&\checkmark($\cong\#12$)\\
&  11 & $t^3ATA^2TATA$&\checkmark($\cong\#9$)\\
&  12 & $t^3aTa^2TATA$&\checkmark\\
&  13 & $t^2aT^2ataTA^2$&\checkmark\\
 & 14 & $t^2AT^2a^2tATA$&\checkmark($\cong\#13$)\\
&  15 & $t^2AtATATATA$&\checkmark\\
&  16 & $t^2ataTATATA$&\checkmark\\
\hline
\end{tabular} 
\end{table}
\hfill\\ 
\begin{table}
\caption{2 generator 1 relator presentations with first Betti number two}
\centering
\begin{tabular}{|rrll|}
\hline
Length&Number&Presentation&Large\\
\hline
16 & 1 & $t^4AtaT^3ATa^2TA$&\checkmark ($\cong 16\#5$ via $a\mapsto at^3$)\\
   & 2 & $t^4ATaT^3Ata^2TA$&\checkmark ($\cong 16\#1$ via $a\mapsto at^4$)\\
   & 3 & $t^3atAT^3AtaTaTA$&\checkmark ($\cong 16\#1$ via $a\mapsto at$)\\
   & 4 & $t^3AtaT^3ATataTA$&\checkmark ($\cong 16\#1$ via $a\mapsto aT$)\\
   & 5 & $t^3ATaT^3AtataTA$&\checkmark (Index 9 subgroup 10 [0,0,0,0])\\  
   & 6 & $t^3aTAT^3AtaTatA$&\checkmark ($\cong 16\#5$ via $a\mapsto aT^2$)\\
   & 7 & $t^3atAT^2AtaT^2aTA$&\checkmark ($\cong 16\#3$ via $a\mapsto at$)\\
   & 8 & $t^2a^2T^2ATa^2tAtATA$&\checkmark (Index 9 subgroup 15 [0,0,0])\\
18 & 1 & $t^4aTaTa^3TA^2TA^3$&\checkmark (Index 13 subgroup 14 [0,0,0])\\
   & 2 & $t^4aT^2AtaTA^2Ta^2TA$&\checkmark (Index 9 subgroup 17 [3,0,0,0])\\
   & 3 & $t^4ATataTaTATaTA^2$&\checkmark ($\cong 18\#5$ via $a\mapsto at$)\\
   & 4 & $t^4ATatATATa^2TaTA$&\checkmark ($a\mapsto at$)\\
   & 5 & $t^3ATat^2a^2TATaTATA$&\checkmark (Index 10 subgroup 24 
                                 [3,18,0,0,0])\\
   & 6 & $t^3ATat^2ATaTa^2TATA$&\checkmark (Index 12 subgroup 18 [0,0,0])\\
   & 7 & $t^3ATa^2T^3atA^2tATa$&\checkmark ($a\mapsto at$)\\
   & 8 & $t^3ATAT^3aTa^2tatA^2$&\checkmark (Index 13 subgroup 11 [0,0,0])\\
   & 9 & $t^3atA^2ta^2T^2aT^2ATA$&\checkmark (Index 7 subgroup 13 [0,0,0,0])\\
   &10 & $t^3atAT^2a^2TA^2taT^2A$&\checkmark (Index 9 subgroup 36 [2,0,0,0])\\
   &11 & $t^3atAT^2aTA^2ta^2T^2A$&\checkmark (Index 10 subgroup 31 [3,0,0,0])\\
   &12 & $t^3atA^2T^2a^2taT^2ATA$&\checkmark ($a\mapsto aT$)\\
   &13 & $t^2aTata^2TA^2taTA^2TA$&\checkmark ($\cong 18\#17$ via 
                                              $a,t\mapsto t,a$)\\ 
   &14 & $t^2aTA^2ta^2TataTA^2TA$&\checkmark ($\cong 18\#20$ via 
                                              $a,t\mapsto t,a$)\\

   &15 & $t^2a^2TAta^2TatA^2TATA$&\checkmark ($a,t\mapsto t,a$)\\
   &16 & $t^2a^2t^2atATATaT^2ATA$&\checkmark ($a,t\mapsto t,a$)\\
   &17 & $t^2aTAt^2Ata^2TATaT^2A$&\checkmark 
(Index 9 subgroup 19 [2,2,0,0,0])\\
   &18 & $t^2aTAt^2AT^2ATata^2TA$&\checkmark ($\cong 18\#6$ via
                                                   $a\mapsto aT$)\\
   &19 & $t^2AT^2aTat^2AtaTaTA^2$&\checkmark ($\cong 18\#16$ via
                                               $a\mapsto at^2$)\\
   &20 & $t^2aTA^2tat^2aTATaT^2A$&\checkmark ($\cong 18\#17$ via
                                               $a\mapsto aT^2$)\\   
   &21 & $t^2aTATat^2ATatATaTA$&\texttt x\\
\hline
\end{tabular} 
\end{table}
\hfill\\ 
\begin{table}
\caption{Closed census hyperbolic 3-manifolds}
\centering
\begin{tabular}{|rlll|rlll|}
\hline
\#&Volume&Name&Large&\#&Volume&Name&Large\\
\hline
8 & 1.4140 & m003(-3,4)&\checkmark (Ar)&
15& 1.5831 & m007(4,1)&\checkmark (Ar)\\
31& 1.8854 & m006(-1,3)&\checkmark (12)&
117&2.4903 & m023(-6,1)&\checkmark (14)\\
210&2.8281&m206(1,2)&\checkmark (Ar)&
244&2.9545&m249(3,1)&\\
295&3.0805&m117(-4,3)&\checkmark (14)&
407&3.2424&m322(2,1)&\\
493&3.3910&m223(-4,1)&&
526&3.4147&m181(-5,1)&\\
572&3.4644&m293(-3,1)&&
674&3.5817&s663(-3,1)&\checkmark (Ar)\\
686&3.5899&m293(1,3)&&
722&3.6281&m249(5,1)&\\
731&3.6360&m238(4,3)&&
932&3.7842&m285(-5,2)&\\
985&3.8418&m322(1,3)&&
1255&4.0188&s663(1,2)&\checkmark (Ar)\\
1310&4.0557&s348(-5,1)&\checkmark (12)&
1324&4.0597&s912(0,1)&\checkmark (Ar)\\
1367&4.0761&s869(-3,1)&&
1407&4.0982&s667(-1,2)&\\
1597&4.1925&s663(-4,1)&&
1728&4.2438&s481(-1,4)&\\
2027&4.3769&m392(-3,4)&&
2070&4.3976&s495(2,3)&\\
2376&4.4954&s705(-4,1)&&
2425&4.5091&v2623(3,1)&\\
2698&4.5853&s645(1,4)&&
2751&4.6044&s663(-5,2)&\\
2793&4.6166&s958(3,1)&&
2834&4.6265&s646(3,4)&\checkmark (Ar)\\
2839&4.6278&s723(4,1)&&
2974&4.6626&s673(5,2)&\\
3095&4.6920&s686(5,1)&&
3210&4.7232&s932(-3,1)&\\
3552&4.7874&v1721(1,4)&\checkmark (12\#3)&
3578&4.7969&v2380(2,3)&\\
3702&4.8309&v1534(5,2)&&
3763&4.8511&s828(-4,3)&\checkmark (11\#1)\\
3801&4.8550&s932(3,1)&&
4035&4.9068&s855(-3,2)&\checkmark (Ar)\\
\hline
\end{tabular}
\end{table}

\begin{table}
\centering
\begin{tabular}{|rlll|rlll|}
\hline
\#&Volume&Name&Large&\#&Volume&Name&Large\\
\hline
4294&4.9664&s872(-1,4)&&
4309&4.9708&s932(-1,3)&\\
4332&4.9751&s932(-3,2)&&
4382&4.9875&s869(1,3)&\\
4471&5.0010&s860(6,1)&&
4599&5.0348&s908(-2,3)&\\
4627&5.0420&v2403(-3,1)&&
4705&5.0578&v2689(4,1)&\\
4745&5.0662&v2217(1,3)&&
4766&5.0730&s932(3,2)&\\
4808&5.0793&v2600(1,3)&&
4847&5.0864&v2478(-4,1)&\\
4899&5.0976&v2568(-1,3)&&
4924&5.1030&s903(-4,1)&\\
4948&5.1080&s907(-4,3)&&
4989&5.1158&v2507(-4,1)&\\
5212&5.1517&v2986(4,1)&&
5262&5.1626&v2920(-4,1)&\\
5292&5.1692&v2315(3,4)&&
5358&5.1793&v2380(-3,2)&\\
5462&5.1979&v2910(-3,2)&&
5556&5.2110&v2438(-5,1)&\\
5574&5.2164&v2344(6,1)&&
5588&5.2181&v2493(5,1)&\\
5617&5.2200&v2876(3,2)&&
5619&5.2204&v2735(-3,2)&\\
5825&5.2515&v2439(1,4)&&
5965&5.2706&v3247(-1,3)&\\
5966&5.2706&v2923(-2,3)&&
5984&5.2727&v2507(-1,4)&\\
6217&5.3072&v2473(-1,3)&&
6423&5.3352&v2623(6,1)&\\
6525&5.3544&v2813(1,3)&&
6572&5.3593&v2919(3,2)&\\
6581&5.3615&v3191(1,3)&&
6633&5.3700&v2600(-4,3)&\\
6751&5.3866&v3147(3,2)&&
6763&5.3882&v3132(2,3)&\\
6848&5.4000&v2897(-1,3)&&
6900&5.4071&v2817(-2,3)&\\
6972&5.4194&v2660(1,4)&&
7020&5.4245&v2661(3,4)&\\
7061&5.4301&v2919(-5,1)&&
7369&5.4734&v3246(1,3)&\\
\hline
\end{tabular}
\end{table}

\begin{table}
\centering
\begin{tabular}{|rlll|rlll|}
\hline
\#&Volume&Name&Large&\#&Volume&Name&Large\\
\hline
7370&5.4734&v3244(1,3)&&
7407&5.4790&v2734(3,4)&\\
7564&5.5022&v2897(4,3)&&
7644&5.5142&v3132(4,1)&\\
7713&5.5261&v2828(-2,3)&&
7731&5.5287&v2813(-4,3)&\\
7798&5.5406&v3075(6,1)&&
7907&5.5617&v2817(2,3)&\\
8182&5.6070&v3262(-3,2)&&
8217&5.6137&v3015(4,3)&\\
8275&5.6205&v3110(3,2)&&
8391&5.6370&v3245(-5,1)&\\
8504&5.6552&v3358(-1,3)&&
8655&5.6765&v3257(-5,1)&\\
8874&5.7096&v3347(-4,1)&&
8971&5.7261&v3268(-4,1)&\\
9050&5.7392&v3195(-6,1)&&
9124&5.7510&v3237(-4,3)&\\
9333&5.7853&v3105(-5,2)&&
9369&5.7914&v3398(5,1)&\\
9380&5.7937&v3229(4,1)&&
9397&5.7962&v3137(-1,4)&\\
9438&5.8038&v3189(4,3)&&
9456&5.8058&v3193(-1,4)&\\
9578&5.8245&v3397(-2,3)&&
9675&5.8432&v3252(1,3)&\\
9731&5.8530&v3245(-6,1)&&
9861&5.8746&v3417(-3,1)&\\
9935&5.8868&v3289(-5,2)&&
9989&5.8990&v3534(3,2)&\\
10017&5.9062&v3486(3,2)&&
10030&5.9093&v3347(-2,3)&\\
10035&5.9104&v3451(3,1)&&
10049&5.9154&v3280(-5,2)&\\
10062&5.9175&v3445(5,1)&&
10120&5.9322&v3499(-3,2)&\\
10195&5.9498&v3340(-5,2)&&
10356&5.9961&v3540(-3,1)&\\
10540&6.0561&v3541(-3,2)&&
10544&6.0577&v3418(-1,3)&\\
10671&6.1011&v3462(-2,3)&&
10922&6.2144&v3502(2,3)&\\
10951&6.2275&v3500(1,4)&&
11014&6.2712&v3500(5,2)&\\
\hline
\end{tabular} 
\end{table}


\begin{thebibliography}{99}

\bibitem{baum} G.\,Baumslag,
{\it A non-cyclic one-relator group all of whose finite quotients are cyclic},
J. Austral. Math. Soc. {\bf 10} (1969) 497--498.

\bibitem{bmgg} G.\,Baumslag, Charles F.\,Miller III and D.\,Troeger,
{\it Reflections on the residual finiteness of one-relator groups},
Groups Geom. Dyn. {\bf 1} (2007) 209--219.

\bibitem{bh} M.\,Bestvina and M.\,Handel,
{\it Train tracks and automorphisms of free groups},
Ann. of Math. {\bf 135} (1992) 1--51.

\bibitem{brun} A.\,M.\,Brunner,
{\it A group with an infinite number of Nielsen inequivalent one-relator
presentations},
J. Algebra {\bf 42} (1976) 81--84.

\bibitem{brunon} A.\,M.\,Brunner,
{\it On a class of one-relator groups},
Canad. J. Math. {\bf 32} (1980) 414--420.

\bibitem{memap} J.\,O.\,Button,
{\it Mapping tori with first Betti number at least two},
J. Math. Soc. Japan {\bf 59} (2007) 351--370.

\bibitem{meis} J.\,O.\,Button,
{\it Large groups of deficiency 1},
Israel J. Math. {\bf 167} (2008) 111--140.

\bibitem{megg} J.\,O.\,Button,
{\it Largeness of LERF and 1-relator groups},
Available at
\texttt{http://arxiv.org/abs/math/0803.3805}

\bibitem{clr} D.\,Cooper, D.\,D.\,Long and A.\,W.\,Reid,
{\it On the virtual Betti numbers of arithmetic hyperbolic 3-manifolds},
Geom. Topol. {\bf 11} (2007) 2265--2276.

\bibitem{disch} A.\,Dietze and M.\,Schaps,
{\it Determining subgroups of a given index in a finitely presented group},
Canadian J. Math. {\bf 26} (1974) 769--782.

\bibitem{dnth} N.\,M.\,Dunfield and W.\,P.\,Thurston,
{\it The virtual Haken conjecture: experiments and examples},
Geom. Topol. {\bf 7} (2003) 399--441.

\bibitem{edph} M.\,Edjvet, 
{\it The concept of ``Largeness'' in Group Theory},
Ph.\,D. thesis, University of Glasgow (1984).

\bibitem{gerst} S.\,M.\,Gersten and J.\,R.\,Stallings,
{\it Irreducible outer automorphisms of a free group},
Proc. Amer. Math. Soc. {\bf 111} (1991) 309--314.

\bibitem{hg} G.\,Higman,
{\it A finitely generated infinite simple group},
J. London Math. Soc. {\bf 26} (1951) 61--64.

\bibitem{cen} C.\,D.\,Hodgson and J.\,R.\,Weeks,\\
\texttt{ftp://www.geometrygames.org/priv/weeks/SnapPea/SnapPeaCensus/}\\
\texttt{ClosedCensus/ClosedCensusInvariants.txt} (2001).

\bibitem{htrs} D.\,F.\,Holt and S.\,Rees,
{\it Free Quotients of Finitely Presented Groups},
Experiment. Math. {\bf 5} (1996) 49--56.

\bibitem{hw} J.\,Howie,
{\it Free subgroups in groups of small deficiency},
J. Group Theory {\bf 1} (1998) 95--112.

\bibitem{lklr} M.\,Lackenby, D.\,D.\,Long and A.\,W.\,Reid,
{\it Covering spaces of arithmetic 3-orbifolds},
Int. Math. Res. Not. IMRN 2008 no. 12.  

\bibitem{ls} R.\,C.\,Lyndon and P.\,E.\,Schupp,
Combinatorial Group Theory, Springer-Verlag, Berlin-Heidelberg-New
York, 1977.

\bibitem{mccpe} J.\,McCool and A.\,Pietrowski,
{\it On free products with amalgamation of two infinite cyclic groups},
J. Algebra {\bf 18} (1971) 377-383.

\bibitem{mol?} D.\,Moldavanski\u{i} and N.\,Sibyakova,
{\it On the finite images of some one-relator groups},
Proc. Amer. Math. Soc. {\bf 123} (1995) 2017--2020.

\bibitem{sn} Snap, available at
\texttt{http://www.ms.unimelb.edu.au/\textasciitilde snap/} (2003).

\bibitem{sta} J.\,R.\,Stallings,
{\it Topologically unrealizable automorphisms of free groups},
Proc. Amer. Math. Soc. {\bf 84} (1982) 21--24.

\bibitem{venk} T.\,N.\,Venkataramana, 
{\it Virtual Betti numbers of compact locally symmetric spaces},
Israel J. Math. {\bf 166} (2008) 235--238.


\end{thebibliography}
\end{document}